\documentclass[12pt]{article}

\usepackage{amsmath,amssymb,amscd,amsthm,amsfonts}
\usepackage{graphicx,subfigure}
\usepackage{hyperref}
\usepackage{dsfont}
\usepackage{color}
\usepackage{float}

\newtheorem{theorem}{Theorem}[section]

\newtheorem{lemma}[theorem]{Lemma}
\newtheorem{claim}[theorem]{Claim}
\newtheorem{corollary}[theorem]{Corollary}

\newtheorem{definition}[theorem]{Definition}

\newcommand{\N}{\mathbb{N}}
\newcommand{\rr}{\mathds{R}}

\newcommand{\dd}[3]{{{}_{#2}{#1}_{#3}}}

\newcommand{\al}{\alpha}
\newcommand{\be}{\beta}
\newcommand{\ga}{\gamma}
\newcommand{\si}{\sigma}

\DeclareMathOperator{\conv}{conv}

\DeclareMathOperator{\aff}{aff}

\begin{document}

\title{Universality of vector sequences and universality of Tverberg partitions}

\author{Attila P\'or}

%\keywords{sequences of vectors, Tverberg partitions, universality}

%\subjclass[2010]{Primary 52A35, Secondary 52B15}
\maketitle

\begin{abstract} A result of Rosenthal says that for every $q>1$ and $n \in \N$ there is $N \in \N$ such that every sequence of $N$ distinct positive numbers contains, after a suitable translation and possible multiplication by $-1$, a subsequence $a_1,\ldots,a_n$ that is either $q$-increasing (that is, $a_{i+1}>qa_i$ for all $i$) or $1/q$-decreasing ($a_{i+1}<a_i/q$ for all $i$). One of our main theorems extends this result to vector sequences.  This theorem is then used to prove the universality theorem for Tverberg partitions which says that, for every $d$ and $r$, every long enough sequence of points in $\rr^d$ in general position contains a subsequence of length $n$ whose Tverberg partitions are exactly the so called rainbow partitions. 
\end{abstract}

\section{Introduction and main results}\label{section-introduction}

This paper is about sequences of vectors in $\rr^d$ and their universal properties. A property $P$ is called {\sl universal} if for every $n \in \N$ there is $N\in \N$ such that every vector sequence $a_1,\ldots,a_N$ (where the $a_i$s are in general position in $\rr^d$) contains a subsequence of length $n$ that has property $P$. For instance, when $d=1$ the property of being increasing or decreasing is universal according to a theorem of Erd\H os and Szekeres~\cite{ESz} from 1935. Precisely, their result says that any sequence of $n^2+1$ distinct real numbers contains a subsequence of length $n$ that is either increasing or decreasing. Rosenthal's lemma~\cite{Ros} described in the abstract is another universality theorem which extends that of Erd\H os and Szekeres. Another theorem of  Erd\H os and Szekeres from the same paper states that every sequence of $4^n$ 2-dimensional vectors (in general position) contains a subsequence of length $n$ that are in convex position, that is, their convex hull has $n$ vertices. This is the universality of the property ``being in convex position''. The main results in this paper establish further universal properties of vector sequences. To state them some definitions are needed.

We define $a:[n] \rightarrow \rr$ as a sequence of length $n$ where $[n]=\{1,\ldots,n\}$. A $d$-dimensional sequence is a collection of $d$ sequences, that is
$a:[n] \rightarrow \rr^d$. The elements of this $d$-dimensional sequence are the column vectors $a_i =  (\dd{a}{1}{i}, \dots, \dd{a}{d}{i})^T$.
For each $s \in [d]$, $\dd{a}{s}{} = \dd{a}{s}{1}, \dd{a}{s}{2}, \dots$ is its $s$th coordinate sequence. The $d$-dimensional sequence $a$ is {in general position} if any $d$ elements
are linearly independent.

Let $a,a'$ be two ($d$-dimensional) sequences of length $n$ and $n'$ respectively. We say that $a'$ is a subsequence of $a$ if
$n' \leq n$ and there exists a subset $I= \{i_1, \dots, i_{n'} \}$ of $[n]$ such that $i_1 < \dots< i_{n'}$ and
$a'_j = a_{i_j}$.

Throughout the paper we (try to) use the variables consistently, namely $i, j,k,\ell$ for the elements in $[n]$ and $[N]$ and $n,N$ for the length of the sequence, and $s,t,\si,\tau \in [d]$ for the coordinates and $d$ for dimension.

Let $a$ be a $d$-dimensional sequence and $T$ a $d \times d$ invertible matrix. We say that the sequence $Ta : Ta_1, \dots, Ta_n$ is a {\sl linear transformation} of the sequence $a$. Of course,  a coordinate sequence of $Ta$ is a linear combination of the coordinate sequences of $a$.

\begin{definition}
Let $q>1$ be a real number. The sequence $a$ is $q$-increasing if it is positive and for every $i \in [n-1]$ we have $\frac{a_{i+1}}{a_i} > q$.
\end{definition}

\begin{definition}
Let $d>1$ and let $a$ be a $d$-dimensional sequence. We say that $a$ is $q$-pseudo-geometric if every coordinate sequence is positive and
for every $s,t \in [d]$, $s\ne t$ either of the two sequences $\frac{\dd{a}{t}{}}{\dd{a}{s}{}}$ or
$\frac{\dd{a}{s}{}}{\dd{a}{t}{}}$ is $q$-increasing.
\end{definition}

One of our main results says that being $q$-pseudo-geometric is a universal property of vector sequences. Here comes the precise statement.
	
\begin{theorem}\label{th:univ-sequen}
Let $q>1$ be a real number and let $d>1$ be an integer. For every integer $n$ there exists $N = N(d,n,q)$ with the following property.
If $a$ is a $d$-dimensional sequence of length $N$ in general position, then there exists a $d \times d$ invertible matrix $T$
such that $Ta$ has a $q$-pseudo-geometric subsequence of length $n$.
Further more, we can assume that $T$ is a lower triangular matrix.
\end{theorem}

The case $d=2$ was proved by Rosenthal~\cite{Ros} in 1981 in slightly different form, see also~\cite{BuMa} for another proof and applications.
We will come back to Rosenthal's lemma in Section~\ref{sec:proof-univ}.

Our second main result is about Tverberg's theorem~\cite{Tver} which says the following.

\begin{theorem}\label{th:tverberg}%[Helge Tverberg 1966 \cite{Tverberg:1966tb}]
Assume $d,r \in \N$, $d\ge 1,r\ge 2$. Given $n=(r-1)(d+1)+1$ points in $\rr^d$, there is a partition of them into $r$ parts whose convex hulls have a point in common.
\end{theorem}

The case $r=2$ is Radon's theorem~\cite{Rad} from 1920. Then $n=d+2$ so any $d+2$ points in $\rr^d$ can be split into two parts so that their convex hulls intersect. Next we give an example.
Assume that the points come from the moment curve $\gamma(t)=(t,t^2,\ldots,t^d)^T \in \rr^d$ ($t>0$), so we have $d+2$ points $\ga(t_1),\ldots,\ga(t_{d+2})$ with $0<t_1<\ldots<t_{d+2}$. It is well-known (see for instance Gr\"unbaum's book~\cite{Grun} or Matou\v{s}ek's~\cite{Mat}) that there is a unique Radon partition in this case, namely, one set is $P_1=\{\ga(t_i): i \mbox{ odd}\}$ and the other one is $P_2=\{\ga(t_i): i \mbox{ even}\}$. That is, the Radon partition is just two {\sl interlacing} sets, meaning that on the moment curve between two consecutive points of $P_1$ (resp $P_2$) there is a point of $P_2$ (and $P_1$). It is also known that this is the universal Radon partition: for every $d \in \N$ there is $N\in \N$ such that for any $d$-dimensional (general position) vector sequence $p_1,\ldots,p_N$ contains a subsequence $p_{i_1}\ldots,p_{i_{d+2}}$ with $i_1<i_2<\ldots<i_{d+2}$ such that their unique Radon partition is the interlacing sets
$P_1=\{p_{i_j}: j \mbox{ odd}\}$ and $P_2=\{p_{i_j}: j \mbox{ even}\}$. The moment curve shows that this is the unique universal Radon partition. Our second main result shows what the universal Tverberg partitions are. Before stating it further definitions are needed.

Let $n=T(r,d) = (r-1)(d+1)+1$ be the Tverberg number, and assume that $A_1 \cup \dots \cup A_r$ is a {\sl proper partition} of $[n]$ which means that $1\le |A_m|\le d+1$ for all $m\in [r]$. The sets $A_1,\ldots,A_r$ will be called {\sl color classes} or simply {\sl classes} of the partition. We define {\sl blocks} $B_1,\ldots,B_{d+1}$ by
\[B_s=\{(s-1)(r-1)+1,(s-1)(r-1)+2,\ldots,s(r-1)+1\}.
\]
So each block contains $r$ consecutive numbers from $[n]$, and they are almost disjoint: only $B_s$ and $B_{s+1}$ have a point in common, namely $s(r-1)+1$, for all $s \in [d]$.

Let $p_1, \dots, p_n \in \rr^d$ be points in {\sl strong general position} (the definition is given in Section~\ref{sec:tverberg}. The partition $A_1,\ldots,A_r$ of $[n]$ induces a partition of the sequence $p_1,\ldots,p_n$ into $r$ sets $P_m=\{p_i: i\in A_m\}$.

\begin{definition} The proper partition $A_1,\ldots,A_r$ of $[n]$ is called a rainbow partition if $|A_m\cap B_s|=1$ for all $m \in [r]$ and $s \in [d+1]$. The corresponding partition $P_1,\ldots,P_r$ of $P$ is also a rainbow partition.
\end{definition}

Again we try to use the notation $m,\al,\be \in [r]$ for the subscripts of the color classes. We remark here that for $r=2$ a rainbow partition $A_1,A_2$ is two interlacing sets. It is known that for the points $\ga(t_1),\ldots,\ga(t_n)$ on the moment curve the Tverberg partitions are exactly the rainbow partitions if the points $t_1<t_2<\ldots <t_n$ are chosen suitably, namely, heavily increasing. This is an unpublished observation of B\'ar\'any and P\'or, and also of Mabillard and Wagner, see also \cite{PerSi2}. Bukh, Loh and Nivasch~\cite{BuLoNi} prove the analogous statement for the points on the diagonal of the stretched grid, for the definition see their paper. 

Here is the universality theorem for Tverberg partitions.

\begin{theorem}\label{th:univ-tverberg} Given $d,r \in \N$ with $r\ge 2$, there is $N=N(d,r)\in \N$ such that every sequence of length $N$ of $d$-dimensional vectors in strong general position  contains a subsequence $p_1,\ldots,p_n$ with $n=T(r,d)$ whose Tverberg partitions are exactly the rainbow partitions.
\end{theorem}

This has been conjectured by Bukh, Loh and Nivasch~\cite{BuLoNi} and proved there for $d=2$ and in some further special cases. The following question emerged in connection with the results of~\cite{BuNi}. Given a finite set $P \subset \rr^d$ with $|P|$ sufficiently large, are there disjoint subsets $X,Y \subset P$ with $|X|=d+2$, $|Y|=d+1$ such that $\conv Y$ contains 
the Radon point of $X$. Theorem~\ref{th:univ-tverberg} answers this question affirmatively: choose $r\ge 3$, suppose $|P|=N$ where $n=T(r,d)$, write the points of $P$ in a sequence, and let $q_1,q_2,\ldots,q_n$ be the subsequence guaranteed by the theorem. There is a rainbow partition $A_1,\ldots,A_r$ of $[n]$ with 
$|A_1|=\lceil (d+2)/2\rceil,\; |A_2|=\lfloor (d+2)/2\rfloor$, and $|A_m|=d+1$ for all $m>2$. Then the sets in $P$ corresponding to $A_1\cup A_2$ and $A_3$ satisfy the requirement. 

The proof method of Theorem~\ref{th:univ-tverberg} yields the following apparently stronger result.

\begin{theorem}\label{th:univtverberg} Given $d,m,r \in \N$ with $r\ge 2$ and $m\ge n=T(r,d)$, there is $N=N(d,m,r)\in \N$ such that every sequence of length $N$ of $d$-dimensional vectors in strong general position contains a subsequence $p_1,\ldots,p_m$ of length $m$ with the following property.  For every subsequence $p_{i_1},\ldots,p_{i_n}$ its Tverberg partitions are exactly the rainbow partitions.
\end{theorem}

The paper is organized as follows. The proof of Theorem~\ref{th:univ-sequen} is proved in Section~\ref{sec:proof-univ}, with some preparations in Section~\ref{sec:prepseq}.  Dominant $q$-increasing sequences, an important tool in the universality of Tverberg partitions, are presented in Section~\ref{sec:orderdominant}. The linear equation related to Tverberg partitions and the $G_w$ matrices are introduced in Section~\ref{sec:tverberg}. The linear equation formulation implies that the sign patterns of certain determinants decide whether a given partition is Tverberg or not. This leads to the question of finding the dominant monomial in the expansion of these determinants in Sections~\ref{sec:dominant-filling} and ~\ref{sec:finddom}. The proof of Theorem~\ref{th:univ-tverberg} is given in Sections~\ref{sec:mainproof} and ~\ref{sec:mainproofcont}.

\section{Preparations for the proof of Theorem~\ref{th:univ-sequen}}\label{sec:prepseq}

We begin with a few simple observations. Assume $a$ is a $q$-pseudo-geometric sequence. Then the subsequence $a'$ of $a$ where we take each $k$th element, $a_i'=a_{ik}$, is $q^k$-pseudo-geometric.

Again, assume $a$ is $q$-pseudo-geometric and let $s,t\in [d]$ be two different coordinates.
We say that the $t$th coordinate {\sl grows faster} than the $s$th coordinate if 
$\frac{\dd{a}{t}{}}{\dd{a}{s}{}}$ is $q$-increasing. This is a total order on $[d]$.
Therefore there exists a unique permutation matrix $T$ such that in $Ta$ the coordinates 
are already ordered increasingly. That is for every $1\leq t <d$ the sequence 
$\frac{\dd{Ta}{t+1}{}}{\dd{Ta}{t}{}}$ is $q$-increasing.
We say that the sequence $Ta$ is ordered and $q$-increasing, or simply that it is ordered.

The following lemma is a key component in the proof of Theorem~\ref{th:univ-sequen}.

\begin{lemma}\label{l:p-lines}
Let $d>1$ and let $q>3$ be a real number.
Let $a$ be an ordered $d$-dimensional $q$-pseudo-geometric sequence. 
Let $b=\sum \alpha_t \cdot \dd{a}{t}{}$ be a linear combination of the $d$ 
coordinate sequences such that $b$ has $(d-1)$ zero elements
$b_{j_1}=\dots = b_{j_{d-1}}=0$, where for each $t=1,\dots, d-2$ we have $j_t +2< j_{t+1}$. Let $j_0=-\infty$ and $j_d=\infty$.
Then the signs of $\alpha_t$ are alternating and 
for every integer $D>0$ with $j_{t-1}+D<i<j_t-D$ we have $(1-2q^{-D})  |\alpha_t \cdot \dd{a}{t}{i}| <  |\alpha_t \cdot \dd{a}{t}{i}| - (|\alpha_{t-1} \cdot \dd{a}{t-1}{i}| + |\alpha_{t+1} \cdot \dd{a}{t+1}{i}|) < |b_i| < |\alpha_t \cdot \dd{a}{t}{i}|$.
The sequence $b$ only changes sign at the pre-described zeros. 
\end{lemma}

\begin{proof}[Proof of Lemma~\ref{l:p-lines}]
	For every $i \in [n]$ let $\beta(i)\in [d]$ be the smallest $t \in [d]$ such that $|\alpha_t| \cdot \dd{a}{t}{i}$ is the largest element in the set 
	$\{ |\alpha_1| \cdot \dd{a}{1}{i}, \dots, |\alpha_d| \cdot \dd{a}{d}{i} \}$.
	We claim that $\beta(1)\leq \dots \leq \beta(n)$.
	Assume on the contrary that for some $i$ we have $\beta(i) = t > s = \beta(i+1)$. This implies 
	$|\alpha_t| \cdot \dd{a}{t}{i} \geq |\alpha_s| \cdot \dd{a}{s}{i}$ and $|\alpha_t| \cdot \dd{a}{t}{i+1} \leq |\alpha_s| \cdot \dd{a}{s}{i+1}$ so
\[
\frac{\dd{a}{t}{i}}{\dd{a}{s}{i}} \geq \frac{|\alpha_s|}{|\alpha_t|}
\geq \frac{\dd{a}{t}{i+1}}{\dd{a}{s}{i+1}}
\]
which contradicts that the sequence $\frac{\dd{a}{t}{}}{\dd{a}{s}{}}$ is $q$-increasing. \\

If $\beta(j)=s$ then for any $k\neq s$ we have
$|\alpha_k| \cdot \dd{a}{k}{j}\leq |\alpha_s| \cdot \dd{a}{s}{j}$.
Therefore
\begin{align}
q^{(s-k)(i-j)}|\alpha_k| \cdot \dd{a}{k}{i} &\leq |\alpha_s| \cdot \dd{a}{s}{i}
\;\;\;\;\textrm{ for } k<s \textrm{ and } i\geq j \label{eqn:beta1} \\ 
q^{(k-s)(j-i)}|\alpha_k| \cdot \dd{a}{k}{i} &\leq |\alpha_s| \cdot \dd{a}{s}{i}
\;\;\;\;\textrm{ for } k>s \textrm{ and } i\leq j \label{eqn:beta2}
\end{align}
Since $q>3$ we get for $i>j$ that
\begin{align}
& |\alpha_1 \cdot \dd{a}{1}{i} + \dots + \alpha_{s-1} \cdot \dd{a}{s-1}{i}| \leq
|\alpha_1| \cdot \dd{a}{1}{i} + \dots + |\alpha_{s-1}| \cdot \dd{a}{s-1}{i} 
\label{eqn:beta3} \\
\nonumber < \; &
(q^{(-s+1)(i-j)}+\dots+q^{-(i-j)})|\alpha_{s}| \cdot \dd{a}{s}{i}\\
\nonumber < \; & 
\frac{1}{q^{(i-j)}-1} |\alpha_{s}| \cdot \dd{a}{s}{j} < \frac{1}{2} |\alpha_{s}| \cdot \dd{a}{s}{i}
\end{align}
and for $i<j$
\begin{align}
& |\alpha_{s+1} \cdot \dd{a}{s+1}{i} + \dots + \alpha_{d} \cdot \dd{a}{d}{i}| \leq
|\alpha_{s+1}| \cdot \dd{a}{s+1}{i} + \dots + |\alpha_{d}| \cdot \dd{a}{d}{i} 
\label{eqn:beta4} \\
\nonumber < \; &
(q^{-(j-i)}+\dots+q^{(-d+s)(j-i)})|\alpha_{s}| \cdot \dd{a}{s}{i}\\
\nonumber < \; & 
\frac{1}{q^{(j-i)}-1} |\alpha_{s}| \cdot \dd{a}{s}{i} <
\frac{1}{2} |\alpha_{s}| \cdot \dd{a}{s}{i}
\end{align}

Let $t\in [d-1]$. We claim that $\beta$ cannot take the same value on the three consecutive elements around $j=j_t$.
Assume on the contrary that $\beta(j-1)=\beta(j)=\beta(j+1)=s$.
Apply inequality (\ref{eqn:beta3}) with $j=j_t-1$ and $i=j_t$ 
\[
 |\alpha_1 \cdot \dd{a}{1}{j_t} + \dots + \alpha_{s-1} \cdot \dd{a}{s-1}{j_t}| 
 < \frac{1}{2} |\alpha_{s}| \cdot \dd{a}{s}{j_t}
\]
and apply inequality (\ref{eqn:beta4}) with $j=j_t+1$ and $i=j_t$
\[
|\alpha_{s+1} \cdot \dd{a}{s+1}{j_t} + \dots + \alpha_{d} \cdot \dd{a}{d}{j_t}| 
< \frac{1}{2} |\alpha_{s}| \cdot \dd{a}{s}{j_t}
\]

Therefore $b_{j_t}$ cannot be $0$ which is a contradiction.\\

So $\beta$ has to increase by at least one from $j_t-1$ to $j_t+1$ for every $t$.
That is $d-1$ increases which implies that each increase is exactly by one and that $\beta(j_t-1) = t$ and $\beta(j_t+1) = t+1$.
Apply inequality (\ref{eqn:beta3}) with $j=j_t-1$ and $i=j_t$ 
\[
|\alpha_1 \cdot \dd{a}{1}{j_t} + \dots + \alpha_{t-1} \cdot \dd{a}{t-1}{j_t}| 
< \frac{1}{2} |\alpha_{t}| \cdot \dd{a}{t}{j_t}
\]
and apply inequality (\ref{eqn:beta4}) with $j=j_t+1$ and $i=j_t$
\[
|\alpha_{t+2} \cdot \dd{a}{t+2}{j_t} + \dots + \alpha_{d} \cdot \dd{a}{d}{j_t}| 
< \frac{1}{2} |\alpha_{t+1}| \cdot \dd{a}{t+1}{j_t}
\]
So $\alpha_t$ and $\alpha_{t+1}$ have different signs, otherwise $b_{j_t} \neq 0$.

Let $D>0$ be an integer and let $j_{t-1}+D<i<j_t-D$.
The signs of the terms in the sum $\sum_{s=1}^d \alpha_s \cdot \dd{a}{s}{i}$
are alternating, and the terms are increasing in absolute values till 
$\dd{a}{t}{i}$ and decreasing from there on. So  
\begin{align*}
|\alpha_t| \cdot \dd{a}{t}{i}-
 |\alpha_{t-1}| \cdot \dd{a}{t-1}{i}-
 |\alpha_{t+1}| \cdot \dd{a}{t+1}{i}& <
 |\sum_{s=1}^d \alpha_s \cdot \dd{a}{s}{i}| <|\alpha_t| \cdot \dd{a}{t}{i}
\end{align*} 
The statement of the Lemma follows from inequalities (\ref{eqn:beta1}) and (\ref{eqn:beta2}).	
\end{proof}

\section{Proof of Theorem~\ref{th:univ-sequen}}\label{sec:proof-univ}

In his thesis~\cite{Ros} Rosenthal proved the following result, in a slightly different form.
\begin{lemma} 
Let $q>1$ be a real number. For every integer $n$ there exists $N = N(n,q)$ such that
if $a$ is a $2$-dimensional sequence in general position of length $N$ then there exists a linear combination of the two sequences
$\dd{b}{2}{}=\alpha_1 \cdot \dd{a}{1}{} + \alpha_2 \cdot \dd{a}{2}{}$ such that the $2$-dimensional sequence $b = (\dd{a}{1}{}, \dd{b}{2}{})$ has a $q$-pseudo-geometric 
subsequence of length $n$.
\end{lemma}

The next lemma generalizes Rosenthal's and proves Theorem~\ref{th:univ-sequen} by induction.

\begin{lemma} \label{l:stretch}
Let $q>3$ be a real number and let $d>1$ be an integer. For every integer $n$ there exists $N = N(d,n,q)$ with the following property.
Let $a$ be a $d$-dimensional sequence in general position of length $N$ such that the $(d-1)$-dimensional sequence
$(\dd{a}{1}{}, \dots, \dd{a}{d-1}{})$ is $q$-pseudo-geometric.
Then there exists a linear combination of the $d$ coordinate sequences
$\dd{b}{d}{}=\sum_{i=1}^{d} \alpha_i \cdot \dd{a}{i}{}$ such that the $d$-dimensional sequence $b = (\dd{a}{1}{}, \dots, \dd{a}{d-1}{}, \dd{b}{d}{})$ has a $q$-pseudo-geometric 
subsequence of length $n$.
\end{lemma}

\begin{proof}
	Choose $\delta$ such that
	\[
	\frac{(1+\delta)(\delta+2q^{-2})}{1-2q^{-2}}<1
	\]
	for example $\delta = \frac{1}{3}$.
First we define a coloring $\phi$ on the set $\binom{[n]}{d+1}$ as follows.
Let $1 \le i_1 < i_2 < \dots < i_{d+1} \le n$ be $d+1$ different numbers and let $I = \{ i_1, \dots, i_{d+1} \}$.
Let $A_I$ be the $d \times (d+1)$ matrix $A_I = [\dd{a}{t}{i_k}]$ $1\leq t \leq d$ and $1\leq k \leq d+1$. 
Let $w_I=(w_1, \dots, w_{d+1})$ be the $d+1$-dimensional cross product of the rows of $A_I$.
That is, $(-1)^k w_k$ is the determinant of the $d \times d$ matrix that we get by deleting the $k$th column of
$A_I$. The vector $w_I$ is also the unique vector (up to a constant factor) which is orthogonal to every row vector 
of $A_I$ (those are the coordinate sequences restricted to the set $I$).
Since our points are in general position none of the $w_k$ can be zero.
Let $\phi_+(I)$ be color $1$ if $w_1>0$ and $-1$ if $w_1<0$.
For all $0 \le k \leq d$ we look at the ratio $-\frac{w_{k}}{w_{k+1}}$ and if it is smaller then 
$(1+\delta)\frac{\dd{a}{k-1}{i_{k+1}}}{\dd{a}{k-1}{i_k}}$ 
we define $\phi_k(I)$ as $1$ (only for $k\geq 2$),
otherwise if it is larger then
$(1-\delta)\frac{\dd{a}{k}{i_{k+1}}}{\dd{a}{k}{i_k}}$ 
we define $\phi_k(I)$ as $2$ (only for $k\leq d-1$),
otherwise $\phi_k(I)=0$.
Here $\phi_1(I)$ cannot be $1$ and $\phi_d(I)$ cannot be $2$.

Finally let $\phi(I) = (\phi_+, \phi_1(I), \phi_2(I), \dots, \phi_d(I))$. So we use $2\cdot3^d$
colors to color all the $(d+1)$-tuples, and by Ramsey theory we find a subsequence of size $m$ which is monochromatic
if $N$ is large enough in terms of $d,m$ and $q$. Here $m$ will be specified soon. 
Since $\phi_+(I)$ is constant, therefore the determinant of the $d \times d$ matrix of 
any $d$ points has the same sign (after removing the first point). 
This also means that for any $I$ the coordinates of $w_I$ are alternating.
This implies that any linear combination $\sum_{t=1}^d \alpha_t \cdot \dd{a}{t}{}$ of the coordinate sequences
which has $d-1$ zeros has the following two properties.
All elements between two consecutive zeros have the same sign. The two elements adjacent to a zero have different sign (alternating).
That is two elements have the same sign exactly if there are an even number of zeros
between them.

%The $w_k$ are actually determinants of $d\times d $ matrices with alternating signs.
%Since among enough points we can always find $d+1$ where the determinant of every $d$ have the same sign therefore
%the color appearing on all $(d+1)$-tuples must not be $0$.

In fact, there is a unique $k$ such that $\phi_1(I) = \dots = \phi_k(I) = 2$, $\phi_{k+1}(I) = 0$ and $\phi_{k+2}(I) = \dots = \phi_d(I) = 1$. But to finish the proof we only need that at least one of the $\phi_{k+1}(I)$ is $0$.

Let us define $d$ sequences $\dd{c}{0}{}, \dots, \dd{c}{d-1}{}$ as linear combinations of the sequences $\dd{a}{j}{}$.
For $0\leq k <d$ pick the first $k$ $3$-apart elements and the last $(d-1-k)$ $3$-apart elements of $\dd{c}{k}{}$ to be zero.
Then $\dd{c}{k}{}$ is well defined up to a constant factor so we can prescribe one more element.
We claim that one of these sequences will do as $\dd{b}{d}{}$ after deleting the zeros at the start and at the end,
then taking each $K$th element where $K=\lceil 3 \log_{1+\delta} q \rceil$.
The interesting part of the sequence $\dd{c}{k}{}$ is after the first $k$ zeros and before the last $d-1-k$ zeros (except for the first and last element)
Let $i_1, \dots, i_k$ be the position of the first $k$ zeros and $i_{k+3}, \dots, i_{d+1}$ be the position of the last $d-k-1$ zeros.
Let $i_{k+1}, i_{k+2}$ be such that $i_k+3\leq i_{k+1}\leq i_{k+2}-3 \leq i_{k+3}-6$.
Let
$I = \{ i_1, \dots, i_{d+1} \}$. Let $w_I = (w_1, \dots, w_{d+1} \}$ be the orthogonal vector to every coordinate sequence
as above. Then $w_I$ is orthogonal to $\dd{c}{k}{}$ which implies
$\frac{\dd{c}{k}{i_{k+2}}}{\dd{c}{k}{i_{k+1}}} = -\frac{w_{k+1}}{w_{k+2}}$.

If $\phi_{k+1}(I) = 0$ then 
\[
 (1+\delta)\frac{\dd{a}{k}{i_{k+2}}}{\dd{a}{k}{i_{k+1}}} < \frac{\dd{c}{k}{i_{k+2}}}{\dd{c}{k}{i_{k+1}}}
 < (1-\delta)\frac{\dd{a}{k+1}{i_{k+2}}}{\dd{a}{k+1}{i_{k+1}}}	
\]
and the above mentioned subsequence (taking every $K$th element) will work as $\dd{b}{d}{}$, if the length of the sequence $\dd{b}{d}{}$ is at least $n$. 
This is guaranteed by choosing $m=3d+Kn$ as one can see directly. We mention that in the last formula 
we do not have the left hand side inequality if $k=1$, and we do not have the right hand side one if $k=d-1$.

So we can assume that every $\phi_k(I)$ is either $1$ or $2$.
But $\phi_1(I)$ must be $2$ and $\phi_d(I)$ must be $1$, therefore
there exists a $k$ such that $\phi_{k+1}(I) = 2$ and
$\phi_{k+2}(I) = 1$. That is
\[
 (1-\delta) \frac{\dd{a}{k+1}{i_{k+2}}}{\dd{a}{k+1}{i_{k+1}}}<  
 \frac{\dd{c}{k}{i_{k+2}}}{\dd{c}{k}{i_{k+1}}}  \;\;\;\;\;\;\;\;
 \frac{\dd{c}{k+1}{i_{k+3}}}{\dd{c}{k+1}{i_{k+2}}}<
 (1+\delta)  \frac{\dd{a}{k+1}{i_{k+3}}}{\dd{a}{k+1}{i_{k+2}}}. 
\]

We show that this leads to a contradiction. Assume that $\dd{c}{k}{i_{k+1}}>0$ 
and let $-\dd{c}{k+1}{} = \dd{c}{k}{} + \sum_{t=1}^{d-1} \alpha_t \; \dd{a}{t}{}$ 
be such that it is zero at $i_1, \dots, i_k, i_{k+1}, i_{k+4}, \dots, i_{d+1}$.
This is unique as $\sum_{t=1}^{d-1} \alpha_t \cdot \dd{a}{t}{}$ has $(d-2)$ zeros at $i_1, \dots, i_k, i_{k+4}, \dots , i_{d+1}$
and is $-\dd{c}{k}{i_{k+1}}$ at the $i_{k+1}$th position.
By Lemma~\ref{l:p-lines} we know that 
$(1-2q^{-2}) |\alpha_{k+1} \cdot \dd{a}{k+1}{i_{k+1}}| < 
\dd{c}{k}{i_{k+1}} < 
|\alpha_{k+1} \cdot \dd{a}{k+1}{i_{k+1}}| $
and that
$(1-2q^{-2}) |\alpha_{k+1} \cdot \dd{a}{k+1}{i_{k+2}}| < |\sum_{t=1}^{d-1} \alpha_t \cdot \dd{a}{t}{i_{k+2}}| 
< |\alpha_{k+1} \cdot \dd{a}{k+1}{i_{k+2}}| $.

The change from $\dd{c}{k}{}$ to $-\dd{c}{k+1}{}$ is 
$\sum_{t=1}^{d-1} \alpha_t \cdot \dd{a}{t}{}$. So that sum is negative in the $i_{k+1}$th position, and since it has the same sign between $i_k$ and $i_{k+4}$th positions therefore $\dd{c}{k+1}{i_{k+3}}$ is positive. Furthermore
$\dd{c}{k+1}{i_{k+3}} > (1-2q^{-2}) |\alpha_{k+1} \cdot \dd{a}{k+1}{i_{k+3}}|$.
%Again by Lemma~\ref{l:p-lines} w
We also have 
$\dd{c}{k}{i} + (1-2q^{-2}) (\alpha_{k+1} \cdot \dd{a}{k+1}{i})>
\dd{c}{k}{i} + \sum_{t=1}^{d-1} \alpha_t \cdot \dd{a}{t}{i} = c'_i >0 $ 
and therefore
$\dd{c}{k}{i} >(1-2q^{-2}) |\alpha_{k+1} \cdot \dd{a}{k+1}{i}|$. So 
$\frac{\dd{a}{k+1}{i+3}}{(1-2q^{-2})\dd{a}{k+1}{i}} >
\frac{\dd{c}{k}{i+3}}{\dd{c}{k}{i}} = -\frac{w_{k+1}}{w_{k+2}}$.
So we have 
\begin{align*}
\dd{c}{k}{i_{k+1}} &= - \sum_{t=1}^{d-1} \alpha_t \cdot \dd{a}{t}{i_{k+1}}
>(1-2q^{-2}) |\alpha_{k+1} \cdot \dd{a}{k+1}{i_{k+1}}| \\
\dd{c}{k}{i_{k+2}} &> (1-\delta) 
\frac{\dd{a}{k+1}{i_{k+2}}}{\dd{a}{k+1}{i_{k+1}}} \dd{c}{k}{i_{k+1}} >
(1-\delta) (1-2q^{-2}) |\alpha_{k+1} \cdot \dd{a}{k+1}{i_{k+2}}| \\
|\sum_{t=1}^{d-1} \alpha_t \cdot \dd{a}{t}{i_{k+2}}|  & <|\alpha_{k+1} \cdot
\dd{a}{k+1}{i_{k+2}}|. 
\end{align*}
Consequently
\begin{align*}
\dd{c}{k+1}{i_{k+2}} & = -\dd{c}{k}{i_{k+2}} - 
\sum_{t=1}^{d-1} \alpha_t \cdot \dd{a}{t}{i_{k+2}} < (\delta+2q^{-2})
|\alpha_{k+1} \cdot \dd{a}{k+1}{i_{k+2}}| \\
\dd{c}{k+1}{i_{k+3}}  &<
(1+\delta)  \frac{\dd{a}{k+1}{i_{k+3}}}{\dd{a}{k+1}{i_{k+2}}}
\dd{c}{k+1}{i_{k+2}} < (1+\delta) (\delta+2q^{-2}) 
|\alpha_{k+1} \cdot \dd{a}{k+1}{i_{k+3}}|
\end{align*}
This contradicts 
$\dd{c}{k+1}{i_{k+3}} > (1-2q^{-2}) |\alpha_{k+1} \cdot \dd{a}{k+1}{i_{k+3}}|$.
\end{proof}

 We mention that $N=N(d,n,q)$ is, as expected from Ramsey theory, very large.

\section{Dominant $q$-increasing sequence}\label{sec:orderdominant}

Let $a$ be a $(d+1)$-dimensional ordered $q$-increasing sequence.
For every $i,j \in [n]$ and $t \in [d]$ define $f_a(t,i,j)$ as the increase of the fraction of the $(t+1)$st and $t$th sequence from $i$ to $j$, that is,
\[
 f_a(t,i,j) = \frac{    \frac{\dd{a}{t+1}{j}}{\dd{a}{t}{j}}         }{      \frac{\dd{a}{t+1}{i}}{\dd{a}{t}{i}}       } =
 \frac{\dd{a}{t+1}{j} \cdot \dd{a}{t}{i}}{\dd{a}{t}{j} \cdot \dd{a}{t+1}{i} }.
\]
We remark here that the sequence $a$ is ordered and $q$-increasing if and only if for every $i,t$ we have $f_a(t,i,i+1)>q$.
The following properties of the function $f$ are easy to establish.
For every $i,j,k \in [n]$ and $t\in [d]$ we have
\begin{align}
f_a(t,i,k) & = f_a(t,i,j) \cdot f_a(t, j,k) \label{e:prod} \\
f_a(t,i,j) &= f_a(t,i,i+1) \cdot f_a(t,i+1,j) > q \cdot f_a(t,i+1,j) > f_a(t,i+1,j) \label{e:left} \\
f_a(t,i,j) &= f_a(t,i,j-1) \cdot f_a(t,j-1,j) > q \cdot f_a(t,i,j-1) > f_a(t,i,j-1) \\
f_a(t,i,j) &\leq f_a(t,1,n)
\end{align}

We want to control the relation of the following two fractions with respect to the interval $[\frac{1}{q},q]$.
Since they are positive, they are either bigger then $q$, smaller than $\frac{1}{q}$ or between
$\frac{1}{q}$ and $q$.
The two fractions are for every $1\leq i<j<k\leq n$ and for distinct $s,t\in [d]$
\[
 \frac{f_a(t,i,j)}{f_a(s,j,k)} \;\;\;\;\;\;\;\;\;\;\;\; \textrm{ and } \;\;\;\;\;\;\;\;\;\;\;\;
 \frac{f_a(t,j,k)}{f_a(s,i,j)}.
\]

We say that $a$ is {\em left-dominant} if for distinct $s,t\in [d]$ the relation of the first fraction
to the interval $[\frac{1}{q}, q]$ is the same independently of the choice of $i,j,k$.
Observe that if $n\ge 5$ then it can not be inside the interval since by equation (\ref{e:left})
\[
 \frac{f_a(t,1,4)}{f_a(s,4,5)} > q^2 \cdot
 \frac{f_a(t,3,4)}{f_a(s,4,5)}
\]
Similarly, we say that $a$ is {\em right-dominant} if for distinct s,$t\in [d]$
the relation of the second fraction to the interval $[\frac{1}{q}, q]$ is the same independently of the choice of $i,j,k$.
Observe again that if $n\ge 5$ then it can not be inside the interval since
\[
 \frac{f_a(t,2,5)}{f_a(s,1,2)} > q^2 \cdot
 \frac{f_a(t,2,3)}{f_a(s,1,2)}
\]
We say that $a$ is  {\em dominant} if $a$ is both left-dominant and  right-dominant.

\begin{lemma}\label{l:dom}
Let $N>2^{2^{cn}}$ where $c=3^{2d(d-1)}$.
Let $a$ be a $(d+1)$-dimensional ordered $q$-increasing sequence of length $N$.
Then $a$ has a subsequence of length at least $n$ which is dominant.
\end{lemma}

\begin{proof}
For each
$1 \leq i<j<k\leq N$ we color the triple $(i,j,k)$ with two color-vectors of length
$d(d-1)$, that is the number of ordered pairs of $s,t \in [d]$ ($s\ne t$).
The coordinates of the first color vector are $0,1$ or $2$ with respect to the fraction
\[
 \frac{f_a(t,i,j)}{f_a(s,j,k)}
\]
being smaller than $\frac{1}{q}$, between $\frac{1}{q}$ and $q$ or larger than $q$.
Similarly the coordinates of the second color vector are $0,1$ or $2$ with respect to the fraction
\[
 \frac{f_a(t,j,k)}{f_a(s,i,j)}
\]
being smaller than $\frac{1}{q}$, between $\frac{1}{q}$ and $q$ or larger than $q$.
By Ramsey theory we get a monochromatic subsequence of length $n$.
As observed before if $n\ge 5$ we cannot have the color $1$ appear which corresponds to the fraction
being inside the interval $[\frac{1}{q}, q]$.
\end{proof}

From now on we refer to a sequence $a$ as {\em dominant} if it is ordered $q$-increasing and dominant.
Let $a$ be a $(d+1)$-dimensional dominant sequence of length $n$.
For distinct $s,t \in [d]$ either $f_a(t,i,j)$ or $f_a(s,j,k)$ is larger by a factor of $q$
than the other
independently of the choice of $i<j<k$.
Similarly either $f_a(t,j,k)$ or $f_a(s,i,j)$ is larger by a factor of $q$ than the other.
There are four possibilities: The larger value in both cases is the one with $t$, or the one with $s$,
or the one with $i,j$ or the one with $j,k$.

We define four relations $\prec, \sim_r, \sim_l$ and $\sim$ on the set $[d]$ as follows.  \\
For distinct $s,t \in [d]$ let
\begin{itemize}
\item Let $t \prec s$ if for every $i<j<k$
$$ f_a(t,j,k) \cdot q < f_a(s,i,j)  \;\;\;\; \textrm{and} \;\;\;\;\;
 f_a(t,i,j) \cdot q <  f_a(s,j,k) $$
\item Let $t \sim_l s$ and $s \sim_l t$ if for every $i<j<k$
\[
 f_a(t,i,j) > q \cdot f_a(s,j,k)  \;\;\;\; \textrm{and} \;\;\;\;\;
 f_a(s,i,j) > q \cdot  f_a(t,j,k)
\]
\item Let $t \sim_r s$ and $s \sim_r t$ if for every $i<j<k$
\[
 f_a(t,j,k) > q \cdot f_a(s,i,j)  \;\;\;\; \textrm{and} \;\;\;\;\;
 f_a(s,j,k) > q \cdot  f_a(s,i,j)
\]
\end{itemize}
Furthermore let $t \sim s$ if $t \sim_l s$ or $t \sim_r s$ and define $t \sim t$ for every $t$.
Observe that if $t\neq s$ and $t \sim s$ than either $t \sim_r s$ or $t \sim_l s$.

\begin{lemma}\label{l:equiv}
If $n>3$ the relation $\sim$ is an equivalence relation.
In each equivalence class either all elements are right-similar ($\sim_r$) or
all elements are left-similar ($\sim_l$) with each other.
The relation $\prec$ is a total order on the equivalence classes.
\end{lemma}

\begin{proof}
During the proof $i<j<k$ will be three of the four numbers $1,2,3,4$. \\
First we show that $\sim$ is an equivalence relation. It is obviously reflexive and symmetric, so we need to show that it is transitive.
Let $t \sim s \sim v$ and assume that $t \sim_r s$ is right similar, which means that for every $i<j<k$
$$ q f_a(t, i,j) < f_a(s, j, k) \mbox{ and } q f_a(s, i,j) < f_a(t, j, k).$$
We claim that $s \sim_r v$. Assume on the contrary that $s \sim_l v$. Then, using (\ref{e:prod}) and  (\ref{e:left}),
$$ qf_a(s,1,2) < f_a(t,2,3)< qf_a(t,2,3) < f_a(s,3,4)<
 qf_a(s,3,4) < f_a(v,2,3) $$
which is a contradiction.
The relation $t \sim_r v$ follows from
\[
 q \cdot f_a(t,1,2) < f_a(s,2,3) <  q\cdot f_a(s,2,3) <  f_a(v,3,4) <  f_a(v,2,4)
\]
and
\[
 q \cdot f_a(v,1,2) < f_a(s,2,3) <  q\cdot f_a(s,2,3) <  f_a(t,3,4) <  f_a(t,2,4)
\]
again by using (\ref{e:prod}) and  (\ref{e:left}). So $\sim$ is transitive, moreover if two elements are
right-similar in an equivalence class then all pairs are right-similar in that equivalence class.

Now we show that $\prec$ is transitive.
If $t \prec s \prec v$ then
$$q \cdot f_a(t,3,4)<f_a(t,2,4)< q^2 \cdot f_a(t,2,4)< q f_a(s,1 ,2 )
< f_a(v, 2, 3) $$
and
\[
q \cdot f_a(t,1,2)<f_a(t,1,3)< q^2 \cdot f_a(t,1,3)< q f_a(s,3 ,4 )
< f_a(v, 2, 3)
\]
which shows that $\prec$ is transitive.

Finally we show that $\prec$ is well-defined %transitive 
on the equivalence classes.
Let $t \prec s \sim v$. Since $t \sim v$ implies $t \sim s$ therefore $t$ and $v$ cannot be similar. $v \prec t$ would imply $v \prec s$ therefore the relation of $t$ and $v$ must be $t \prec v$.
Similarly $t \sim s \prec v$ implies $t \prec v$.
\end{proof}

Let $t,s \in [d]$ be different. Let $t \vdash s$ if $t \prec s$ or $t \sim_r s$
and $t<s$ or $t \sim_l s$ and $t>s$. The relation $\vdash$ extends $\prec$ to a total order on $[d]$ such that the left similar equivalence classes are ordered decreasingly and the right similar equivalent classes are ordered increasingly.

The following lemma is crucial for the proof of Theorem~\ref{th:univ-tverberg}

\begin{lemma}\label{l:crux}
Let $a$ be a dominant (ordered and $q$-increasing) $(d+1)$-dimensional sequence of length $n$. Assume that $S$ is a non-empty subset of $[d]$,
$\tau \in S$ is the $\vdash$-maximal element in $S$  and for
every $s \in S$ integers $i_s,j_s\in [n]$ are given that satisfy the conditions
\begin{itemize}
\item $|i_\tau-j_\tau|\ge |S|$,
\item$i_s \le i_t$ and $j_s\le j_t$ if $s <t$ and $s,t \in S$,
%\item if $i_s=i_t$ and $j_s=j_t$, then $s=t$.
\end{itemize}
%Assume that $\tau \in S$ is the $\vdash$-maximal element in $S$.
Then
\[
\prod_{s \in S}f_a(s,i_s,j_s) > q \mbox{ if } i_{\tau} <j_{\tau} \mbox{ and }  <\frac 1q \mbox{ if } i_{\tau} > j_{\tau}.
\]
\end{lemma}

\begin{proof}
Observe first that we can assume that $i_\tau < j_\tau$, since switching
each pair $i_s, j_s$ changes the product into its reciprocal.
Further observe that we can assume $i_s > j_s$ for every $s \neq \tau$ since
those are the terms in the product, that are smaller than $1$.

We claim that if $s \neq \tau$ then for any $i_\tau \leq i < j_\tau$
\[
1 < qf_a(s,j_s,i_s) < f_a(\tau, i, i+1)
\] 
{\bf Case 1} when $s < \tau$. 
Then $j_s < i_s \leq i_\tau \leq i$. 
Since $s\vdash \tau$ therefore $s \sim_l \tau$ is not possible.
Therefore $s \prec \tau$ or $s\sim_r \tau$ and 
\[
qf_a(s,j_s,i_s) \leq qf_a(s,j_s,i) < f_a(\tau, i, i+1) \enspace.
\]
{\bf Case 2} when $s > \tau$. 
Then $i+1\leq j_\tau \leq j_s < i_s$. 
Since $s\vdash \tau$ therefore $s \sim_r \tau$ is not possible.
Therefore $s \prec \tau$ or $s\sim_l \tau$ and 
\[
qf_a(s,j_s,i_s) \leq qf_a(s,i+1,i_s) < f_a(\tau, i, i+1) \enspace.
\]
Multiplying $|S|-1$ of these inequalities together we get
\[
q\prod_{\tau \neq s \in S} f_a(s,j_s,i_s) < \prod_{i=i_\tau}^{j_\tau-1} f_a(\tau, i, i+1) = f_a(\tau, i_\tau, j_\tau)
\]
\end{proof}

A sequence $a=(a_1,\ldots,a_n)$ will be called {\em super-dominant} if there is a dominant (and then ordered and $q$-increasing) sequence $b_1,b_2,\ldots,b_{(d+1)n}$ of $(d+1)$-dimensional vectors so that $a_i=b_{i(d+1)}$ for every $i\in [n]$. We will use the following corollary to Lemma~\ref{l:crux} in the proof of Theorem~\ref{th:univ-tverberg}.

\begin{corollary}\label{cor:crux}
Let $a$ be a super-dominant $(d+1)$-dimensional sequence of length $n$ and let $s,t \in [d]$ satisfy $s<t$. Let $i_s\le i_{s+1}\le \ldots \le i_t$ and $j_s \le j_{s+1} \le \ldots \le j_t$ be integers in $[n]$. 
%Assume further that if $u,v \in \{s,s+1,\ldots,t\}$ and $i_u=i_v$ and $j_u=j_v$, then $u=v$.
Let $\tau$ be the $\vdash$-maximal element in $\{s,s+1,\ldots,t\}$. Then
\[
\prod_{u=s}^t f_a(u,i_u,j_u) > q \mbox{ if } i_{\tau} <j_{\tau} \mbox{ and }  <\frac 1q \mbox{ if } i_{\tau} > j_{\tau}.
\]
\end{corollary}

\begin{proof} The three conditions in Lemma~\ref{l:crux} are satisfied because $a$ is super-dominant.
\end{proof}

\section{Tverberg partitions and the $G_w$ matrices}\label{sec:tverberg}

The next definition is from \cite{PerSi} and establishes a stronger property than
general position. The $d$-dimensional sequence $p_1,\ldots,p_n$ is in {\em strong general position} if
for any collection of $r$ pairwise disjoint subsets $A_1, \dots, A_r \subset [n]$
the affine hulls $H_m = \aff \{ p_i \; | i \in A_m \}$ intersect in such a way that their co-dimensions add up, that is
\[
d -\dim( \bigcap_{m=1}^r H_m) = \min (d+1, \sum_{m=1}^r (d-\dim(H_m))).
\]
It follows, in particular, that for a non-proper partition $A_1,\ldots,A_r$ of $[n]$ where $n=(d+1)(r-1)+1$, and with $P_m=\{p_i:i\in A_m\}$, the intersection of the affine hulls of $P_1,\ldots,P_m$ is empty. On the other hand, strong general position implies that for $n=(d+1)(r-1)+1$ and for a proper partition $A_1,\ldots,A_r$ of $[n]$, $\bigcap_1^r \aff P_m$ is a single point, let  $z = (\dd{z}{1}{}, \dots, \dd{z}{d}{})^T$ be this point.
Thus for all $i \in [n]$ let $\alpha_i$ be the unique coefficient such that for all $m\in  [r]$
$$ \sum_{i \in A_m} \alpha_i p_i  = z, \;\;\;\; \sum_{i \in A_m} \alpha_i = 1 $$.
In this linear system we have $r(d+1)$ equations and the same number of variables $\alpha_1, \dots, \alpha_n, \dd{z}{1}{}, \dots, \dd{z}{d}{}$.
The partition $A_1, \dots, A_r$ is a Tverberg partition if all the $\alpha_i$ are positive. An equivalent formulation is the following.
The partition $A_1, \dots, A_r$ is a Tverberg partition if for every $m \in [r]$ all the $\alpha_i$ have the same sign
for all $i \in A_m$.

We write this system of equations in matrix form $M \alpha = b$ and use Cramer's rule to find the sign of the coefficients.
The variables are $\alpha = (\alpha_1, \dots, \alpha_n, \dd{z}{1}{}, \dots, \dd{z}{d}{})^T$.
Here $M$ is the following square matrix of size $r(d+1)$, see Table 1. The rows of $M$ come in consecutive $(d+1)$-tuples, each associated with a color class, $A_m$, say. We will call this set of $d+1$ rows the {\em rows of color} $m$.
Each column of $M$ is a concatenation of $r$ $(d+1)$-vectors, one for each color class. Let $i \in [n]$ and let $i \in A_m$. Then column $i$ of $M$ is the vector where each such
$(d+1)$-vector is the zero vector except the one of color class $m$, which is the vector $(1, p_i)^T = (1, \dd{p}{1}{i}, \dd{p}{2}{i}, \ldots, \dd{p}{d}{i}, )^T\in \rr^{d+1}$.
The last $d$ columns of $M$ are related to $z$, that is, if $i=(r-1)(d+1)+1+j$, then the $i$th column of $M$ is the concatenation of $r$ copies of the $(d+1)$-vector
$(0, \dots, 0, -1, 0, \dots, 0)^T$ where the only non-zero coordinate $-1$ is in the $(1+j)$th position. Similarly, the right hand side $b$ is the concatenation of $r$ copies of the $(d+1)$-vector $(1,0, \dots, 0)^T$. For simpler writing we define (for every $i$) $\dd{p}{0}{i}$ as the $0$th coordinate of the point $p_i$ to be $1$. Let $M_{\ell}$ be the matrix
where we replace column $\ell$ of $M$ by $b$.

\begin{table}[t]\label{fig:1}
\begin{center}
\begin{tabular}{|c|c|c|c|c|}
  \hline
     1 1 $\dots$ 1 & & & & \\
  \hline
   & & & & \\
   $A_1$ & & & & \hspace*{.1in} -I
                           \hspace*{.1in}  \\
   & & & & \\
  \hline
    & 1 1 $\dots$ 1 & & & \\
  \hline
   & & & & \\
  &  $A_2$ & & & -I \\
   & & & & \\
  \hline
  & & $\ddots$ & & \\
  \hline
     & & & 1 1 $\dots$ 1  & \\
  \hline
   & & & & \\
   & & & $A_r$ & -I \\
   & & & & \\
  \hline
\end{tabular}
\end{center}
\caption{The matrix $M$, the empty regions indicate zeros}
\end{table}

Figure Table 1 shows the essence of the matrix $M$ where the columns are rearranged so that columns from the same color class come consecutively.

Because of Cramer's rule, the signs of the determinants of the matrices $M_{\ell}$, $\ell \in [n]$ decide if the given partition is Tverberg: if all the signs are the same, then it is a Tverberg partition, otherwise it is not.
The determinant of $M_{\ell}$ is a polynomial with the coordinates of the points as variables.
As $\det(M_{\ell})$ is a sum of monomials we want to understand how each of these monomials look like.
Each monomial is a product of entries in a {\em transversal} of the matrix $M_{\ell}$, which is a set of entries exactly one from each row and each column.
Since we take one from each column, every point contributes a coordinate, the $t$th (maybe the $0$th coordinate which is 1)
to the product. Every column related to $z$ contributes a $-1$ or $0$ to the product in one of the rows of some color $m \in [r]$.

Some of these monomials are identically zero. From now on we are only interested in non-zero monomials.
For each non-zero monomial $w$ we define an auxiliary matrix $G_w$ of size $(d+1) \times r$ where column $m$ corresponds to the color class $A_m$ of the partition
and row $t$ corresponds to the $t$th coordinate. We denote by $G_w(t,m)$ the entry sitting in row $t$ and column $m$ of $G_w$. Note that $G_w$ has a $0$th row.
We fill out the entries of the matrix with the numbers $i\in [n]\setminus\{\ell\}$ and $\dd{z}{0}{}, \dd{z}{1}{}, \dots, \dd{z}{d}{}$ the following way.
If $i \in [n]\setminus \ell$ and $i \in A_m$ and the point $p_i$ appears with its $t$th coordinate in $w$, then $G_w(t,m)=i$.
If the $1$  from column $\ell$ that appears in $w$ comes from the rows of $m$, then $G_w(0,m)=\dd{z}{0}{}$.
If the $-1$ from the column corresponding to $\dd{z}{t}{}$ that appears in $w$ is from the rows of color $m$, then $G_w(t,m)=\dd{z}{t}{}$.

It is obvious that such a $G_w$ matrix has one $z$ entry in every row. Also, column $m$ contains exactly $d+1-|A_m|$ $z$ entries except when $\ell \in A_m$,
in which case column $m$ contains $d+1-(|A_m|-1)$ $z$ entries. It is easy to recover the monomial $w$ from such a $G_w$ matrix: $\ell \in [n]$ not appearing in the matrix is the subscript of the column where the right hand side vector $b$ sits, the column of entry $\dd{z}{0}{}$ in row $0$ shows where the $+1$ factor from column $\ell$ of $M_{\ell}$ comes from. Similarly the column of entry $\dd{z}{t}{}$ in row $t$ shows where the $-1$ factor in column of $\dd{z}{t}{}$ in $M_{\ell}$ comes from. Finally, if $i \in [n]\setminus\{\ell\}$ appears as entry $(t,m)$ in $G_w$, then $\dd{p}{t}{i}$ appears in $w$.

We say that a partial filling of the matrix $G$ of size $(d+1) \times r$ with the numbers $i \in [n]\setminus \{\ell\}$ is valid
if each entry is in the correct column with respect to the partition, and there is exactly one unfilled entry in each row.
We claim that each valid filling of $G$ corresponds to a unique non-zero monomial $w$ (and so a unique transversal) that appears in the expansion of $\det (M_{\ell})$. All we need to do is put
$\dd{z}{t}{}$ in row $t$ in the empty slot. Thus non-zero monomials appearing in $\det (M_{\ell})$ are in one-to-one correspondence with valid partial fillings of $G$.

We mention that each non-zero monomial appears in $\det(M_{\ell})$ with a $\pm1$ coefficient, depending on the determinant of the underlying transversal.

\section{Dominant fillings}\label{sec:dominant-filling}

Let $p$ be a $d$ dimensional sequence of length $n=(r-1)(d+1)+1$.
Let $M$ and $M_{\ell}$ be defined as in Section~\ref{sec:tverberg}.
Assume that $Ta=(1,p)$ for some permutation matrix $T$ and $a$ is a
super-dominant (and then ordered and $q$-increasing) sequence of length $n$ and
$A_1, \dots, A_r$ is a proper partition of $[n]$.

Let $w$ and $w'$ be two monomials from the expansion of $M_{\ell}$. We say that $w$ {\em dominates} $w'$ if $q\cdot |w'| < |w|$.
If $w$ dominates all other monomials in the expansion, then we call it {\em dominant}.

\begin{theorem}\label{th:domm} Every monomial in the expansion of $\det(M_\ell)$,  except for the largest one, is dominated by some other monomial. Thus there is always a dominant monomial.
The sign of the dominant monomial is the same as that of $\det(M_{\ell})$ provided $q>(r(d+1))!$.
\end{theorem}

We state and prove two lemmas that are needed for this theorem. We begin by describing how one can find the dominant filling of $G$. This is based on the relations $\prec, \sim_l, \sim_r$.

We can assume that $(1,p)$ is ordered apart from the $1$-sequence, which might be some other coordinate different from the first one.

The matrix $G_w$ is filled with {\em elements}, that is integers from $[n] \backslash \{\ell\}$ and with $z$-{\em entries}, that is with $\dd{z}{s}{}$, $s=0,1,\ldots,d$.
So $G$ is filled with $n-1=(r-1)(d+1)$ elements of $[n]$ and $d+1$ $z$ entries.

First we show that if the elements in a column of $G_w$ are not ordered increasingly, then $w$ is dominated.
Let $w$ be a monomial in $\det(M_{\ell})$ Let $i,j$ be elements of $[n]$
in the same column of $G_w$ in the wrong order. That is $i$ is in row $s$ and $j$ be in row $t$ and $s<t$ but $i>j$.

Swapping $i,$ and $j$ and keeping all other entries of $G_w$ the same we get $G_{w'}$, another valid filling. Then
$$\frac{w'}{w} = \prod_{u=s}^{t-1} f_a(u,j,i)>q$$
which shows that $w$ is dominated by $w'$.
This means that the elements of $[n]$ have to be ordered increasingly in each column otherwise $w$ is dominated. We will call this the {\em increasing order in columns} rule.
From now on we assume this property about every $w$ we work with.
This also means that every $w$ is described by the positions of the $z$ entries in $G_w$.
Recall that every row contains exactly one $z$ entry,
and that the number of $z$ entries in column $m$ is $(d+1)-|A_m|$ except when $\ell \in A_m$, and then there are $d+1-(|A_m|-1)$ $z$ entries in column $m$.

To show that there is a dominant monomial we state and prove a lemma. Assume that $s,t \in [d]$ and $s<t$ and that $G_w(s,\al)=\dd{z}{s}{}$ ($\dd{z}{s}{}$ is in row $s$, of course) and  $G_w(t,\be)=\dd{z}{t}{}$. Let the $z$ entries in $G_{w''}$ be filled the same way as $G_w$ except that $\dd{z}{s}{}$ is in column $\be$ $\dd{z}{t}{}$ is in column $\al$.
We say that {\em switching} these two $z$ entries in $w$ gives $w''$, or that $w''$ is the $z$ {\em switch} of $w$.

Because of the increasing order in columns rule, the elements in column $\al$ between rows $s$ and $t$ have to move up, and the elements in column $\be$ between rows $s$ and $t$ have to move down. More precisely, for $u \in \{s,s+1,\ldots,t-1\}$ let $i_u$ be the smallest element in column $\al$ below row $u$ in $G_w$, and let $j_u$ be the largest element in column $\be$ above row $u+1$ in $G_w$. Note that $i_{t-1}$, resp. $j_s$ are welldefined as $G_w(t,\al)$ and $G_w(s,\be)$ are elements in $[n]\setminus \{\ell\}$. This means that the numbers $i_u,j_u$ are all welldefined.
% and $i_u=i_{u+1}$ happens iff $\dd{z}{u}{}$ is in cell $(u,\al)$ in $G_w$, and similarly for $j_u=j_{u+1}$.

\begin{lemma}\label{l:switch} Under these conditions let $\tau$ be the $\vdash$-maximal element in $\{s,s+1,\ldots,t-1\}$. Then $w''$ dominates $w$ if $i_{\tau}<j_{\tau}$ and $w$ dominates $w''$ if  $i_{\tau}>j_{\tau}$.
\end{lemma}

\begin{proof} It is clear from their definition that $i_s\le i_{s+1} \le \ldots \le i_{t-1}$ and $j_s\le j_{s+1}\le \ldots \le j_{t-1}$. A direct computation shows that
\begin{eqnarray*}
\frac{w''}{w} &=& \frac{\dd{p}{s+1}{i_s}\dd{p}{s+2}{i_{s+1}}\ldots\dd{p}{t}{i_{t-1}}\dd{p}{s}{j_s}\ldots\dd{p}{t-1}{i_{t-1}}}{\dd{p}{s}{i_s}\dd{p}{s+1}{i_{s+1}}\ldots\dd{p}{t-1}{i_{t-1}}\dd{p}{s+1}{j_s}\dd{p}{s+1}\ldots\dd{p}{t}{j_{t-1}}}\\
&=& \prod_{u=s}^{t-1} f_a(u,i_u,j_u),
\end{eqnarray*}
where only the different variables are shown (ignoring the $\pm 1$s in $w$ and $w''$).

Further, $i_u \ne j_u$ since $i_u \in A_\al$ and $j_u \in A_\be$.
Then Corollary~\ref{cor:crux} applies and finishes the proof.
\end{proof}

Another way if stating this result is the following. Assume $s,t \in [r]$, $s<t$, $\tau$ is the $\vdash$-maximal element in $\{s,s+1,\ldots,t\}$,
$G_w$ satisfies the increasing order rule, and $\dd{z}{s}{}=G_w(s,\al)$ and $\dd{z}{t}{}=G_w(t,\beta)$ with $\al\ne \be$. Let $i$ be the smallest element in column $\al$ below row $\tau$ and let $j$ be the largest element in or above row $\tau$ in column $\be$.
\begin{equation}
\mbox{ If under these conditions } i<j, \mbox{ then }w \mbox{ is dominated by }w''. \tag{*}
\end{equation}

A certain converse to this statement also holds. Namely, let $\tau$ be the $\vdash$-maximal element in $[d]$ and $s\le \tau <t$, and otherwise the previous conditions hold. If for every such $s,t,\al,\be,i,j$,
\begin{equation}
 i>j,\mbox{ then  no } z\mbox{ switch including row } \tau \mbox{ gives a }w'' \mbox{ dominating }w. \tag{**}
\end{equation}

Assume now that $w,w'$ are two different monomials, meaning that the positions of the $z$ entries are different and both obey the increasing order in columns rule.

\begin{lemma}\label{l:maxdom} Under these conditions either in $w$ or in $w'$ one can switch two $z$ entries to get $w''$ which dominates that monomial.
\end{lemma}

\begin{proof} Let $L(w,m,t)$ be the number of $z$ entries in $G_w$ in the first $t$ rows in column $m$.
If for every $t\in [d]$ and $m \in [r]$ we have $L(w,m,t) = L(w',m,t)$ then $w=w'$.
Let $\tau$ be maximal with respect to $\vdash$ such that there exists $m \in [r]$ such that $L(w,m,\tau) \neq L(w',m,\tau)$.
Then there exists $\al,\be \in [r]$ such that $L(w,\al,\tau) < L(w',\al,\tau)$ and $L(w,\be,\tau) > L(w',\be,\tau)$, since 
$\sum_{m=1}^r L(w,m,\tau) = \sum_{m=1}^r L(w',m,\tau) = \tau$.

Let $i,j,i^*,j^*$ be elements in $[n]$ such that
$i$ is the largest in $G_w$ in column $\al$ in the first $\tau$ rows,
$j$ is the smallest in $G_w$ in column $\be$ below row $\tau$,
$i^*$ is the largest in $G_{w'}$ in column $\be$ in the first $\tau$ rows,
$j^*$ is the smallest in $G_{w'}$ in column $\al$ below row $\tau$,
see Table 2 where $\dd{z}{s}{}$ and $\dd{z}{t}{}$ denote some $s$-coordinates with $s<t$ but unspecified otherwise.

Equivalently one could say that $i$ is the $\tau - L(w,\al,\tau)$-th element of $A_\al$, 
$j^*$ is the $\tau+1 - L(w',\al,\tau)$-th element of $A_\al$, 
$i^*$ is the $\tau - L(w',\be,\tau)$-th element of $A_\be$ and
$j$ is the $\tau+1 - L(w,\be,\tau)$-th element of $A_\be$.

\begin{table}[t]\label{fig:2}
	\begin{center}
		\begin{tabular}{c|c}
			$\al$ & $\be$ \\
			\hline
			 & $\dd{z}{s}{}$ \\
			 $i$ &  \\
			\hline
			 & $j$ \\
			  $\dd{z}{t}{}$ & \\
		\end{tabular} \hspace{1in}
		\begin{tabular}{c|c}
			$\al$ & $\be$ \\
			\hline
			& $i^*$ \\
			& $\dd{z}{s}{}$ \\
			\hline
			$\dd{z}{t}{}$ & \\
			$\dd{z}{t}{}$ & \\
			$j^*$ &  \\
		\end{tabular}
	\end{center}
	\caption{sketch of columns $\al,\be$ in the matrices $G_w$ (left) and $G_{w'}$ (right)}
\end{table}

Since $L(w,\be,\tau)> L(w',\be,\tau)$ we have $j < i^*$ and since $L(w,\al,\tau)< L(w',\al,\tau)$ we have $j^* < i$. Therefore
the minimal element among $i,i^*,j,j^*$ is either $j$ or $j^*$.
Assume it is $j$, then $j<i$. There is a last $z$ element $\dd{z}{s}{}$ in column $\be$
in the first $\tau$ rows in $G_w$ and a first $\dd{z}{t}{}$ in column $\al$ below row $t$.
Switching $\dd{z}{s}{}$ and $\dd{z}{t}{}$ in $G_w$ gives $G_{w''}$ such that $w''$ dominates $w$, according to (*).
The case when $j^*$ is the smallest leads to an analogous $G_{w''}$ with $w''$ dominating $w'$.
\end{proof}

\begin{proof}[Proof of Theorem~\ref{th:domm}]
Applying the last lemma when $w$ is the largest and $w'$ is the second largest monomial (in absolute value) shows $w$ dominates $w'$.
Consequently $w$ dominates every other monomial $w''$ and so it is the dominant monomial. Further, if $q>((d+1)r)!$, the number of monomials in the expansion of $\det (M_{\ell})$,
then the sign of $\det (M_{\ell})$ coincides with that of $w$.
\end{proof}

Our the next target is to find the dominant monomial.

\section{How to find the dominant filling of G}\label{sec:finddom}

Given a proper partition with color classes $A_1,\ldots,A_r$ of $[n]$ where $n=(r-1)(d+1)+1$ and $\ell \in [n]$, and a super-dominant $q$-increasing sequence $(1,p)$ of length $n$, we want to find the dominant monomial of $\det(M_{\ell})$. We only need to find the positions of the $z$-coordinates. We will find them by recursion on $d$.

As usual, let $\tau$ be maximal in $[d]$ with respect to $\vdash$.
 Splitting the set $[n]\setminus \{\ell\}$ into sets $X$ and $Y$ 
 %with $|X|=\tau(r-1)$ and $|Y|=(d+1-\tau)(r-1)$ 
 defines a splitting of each color class $A_m$ into two pieces: $A_m^X=A_m\cap X$ and $A_m^Y=A_m\cap Y$ for every $m\in [r]$. We define the excess of $X$ in $A_m$ as $e(X,m)=||A_m^X|-\tau|_+$ and the excess of $Y$ in $A_m$ as  $e(Y,m)=||A_m^Y|-(d+1-\tau)|_+$, and the excess of $X$ and $Y$ as $e(X)=\sum_1^re(X,m)$ and $e(Y)=\sum_1^re(Y,m)$. We want to find a splitting with the following properties.
\begin{itemize}
\item[(a)] Every element of $A_m^X$ is smaller than any element of $A_m^Y$ (for every $m$).
\item[(b)] $e(X,m) =e(Y,m)=0$ (for every $m$).
\item[(c)] For every distinct $\al,\be \in [r]$ such that $|A_{\al}^X|<\tau$ and $|A_{\be}^Y|<d+1-\tau$ every element of 
$A_{\be}^X$ is smaller then any element in $A_{\al}^Y$.
\item[(d)] $|X| = \tau (r-1)$ (and therefore $|Y| = (d+1-\tau)(r-1)$).
\end{itemize}

\begin{lemma}\label{l:split} There is always a splitting $X,Y$ such that the above conditions are satisfied, and it can be found in a process of three steps. In fact that splitting is unique and corresponds to the dominant $w$ of $\det(M_\ell)$ such that $X$ is the set of all elements in the first $\tau$ rows of $G_w$.
\end{lemma}

\begin{proof} During the process condition (a) will always be fulfilled.
We begin with the splitting $X,Y$ of the given sizes,
Condition (d), such that every element in $X$ is smaller than any element of $Y$.
So $X$ consists of the smallest $\tau(r-1)$ elements of $[n]\setminus \{\ell\}$ and $Y$ of the rest. If they satisfy condition (b) then we are done:
condition (c) is fulfilled as every element of $A_\al^X$ is smaller then every element of $A_\be^Y$ independent of the choice of $\al, \be$.

So assume condition (b) fails, some $A_m^X$ has size larger than $\tau$ or some $A_m^Y$ has size larger than $d+1-\tau$. The second step fixes this and generates a splitting where condition (b) holds: keep the smallest $\tau$ elements in $A_m^X$ and put the rest in $A_m^Y$ and similarly when $|A_m^Y|> d+1-\tau$. After this exchange we get a new splitting $X',Y'$ of $[n]\setminus \ell$ and $|X'|=|X|-e(X)+e(Y)$ and $|Y'|=|Y|-e(Y)+e(X)$. It is easy to check that condition (c) still holds: if $|A_{\al}^{X'}|< \tau$, then we possibly pushed some elements from $A_{\al}^Y$ up to $A_{\al}^X$ so every $i\in A_{\al}^{Y'}$ is in $Y$. And similarly, if $|A_{\be}^{Y'}|<d+1-\tau$, then every $j\in A_{\be}^{X'}$ is in $X$ and thus every element of $A_\be^{X'}$ is smaller than any element of $A_\al^{Y'}$.
Thus if $|X'|=\tau(r-1)$ and $|Y'|=(d+1-\tau)(r-1)$ then condition (d) holds as well and we are done again.

Finally if $|X'|\ne |X|$, then $|X'|-|X|=|Y|-|Y'|>0$, say, then we have to ``push'' some ($|X'|-|X|$) elements of $X'$ down to $Y'$. An element in column $m$ from $X'$ can be pushed down to $Y'$ if there is room there, that is $|A_m^{Y'}|<d+1-\tau$. One by one we push down the largest ``pushable'' element form $X'$ as long as the number of elements that remains in $X'$ reaches $\tau(r-1)$. We check the largest pushable element before each push, since some pushable elements might become non-pushable as the new set $Y'$ already contains $(d+1-\tau)$ elements of $A_m$. Once this is achieved, we have a new splitting $X'',Y''$ that satisfies conditions (a),(b) and (d). Condition (c) is also satisfied.
Assume $\al,be$ are as in condition (d) after all the pushes.
That implies that all the elements of $A_\be^{X''}$ were pushable during the process, and therefore are elements of $X'$. Let $i$ be the smallest 
element of $A_\al^{Y''}$. If $i$ is an element of $A_\al^{Y'}$ then
$\al,be$ were as in condition (c) before the pushes, so it holds.
If $i$ is not an element of $A_\al^{Y'}$ then we had to push it down during the process, and therefore is larger then any element in 
$A_\be^{X''}$.

The case when $|X'|-|X|<0$ is completely analogous, replacing pushing down, by pushing up.
\end{proof}

We explain now how this lemma helps to find the dominant filling. The partition $A_1,\ldots,A_r$ determines how many $z$ coordinates are in column $m$,
namely $d+1-|A_m|$ if $\ell \notin A_m$ and  $d+1-(|A_m|-1)$ if $\ell \in A_m$. The splitting $X,Y$ in Lemma~\ref{l:split} determines how many $z$-coordinates are in the first $\tau$ and in the last $d+1-\tau$ rows if column $m$ (for every $m$). So recursion on $d$ ends with a special partition for every row: $r-1$ single elements and the empty set. The column where the empty set is contains the $z$ entry for this row. The resulting $w$ is the dominating monomial since, if it were not, then a suitable $z$-switch would give a $w'$ dominating it. But in view of condition (c) of Lemma~\ref{l:split} and (**) from Section~\ref{sec:dominant-filling} no such $z$-switch exists.

\section{Proof of Theorem~\ref{th:univ-tverberg}, first part}\label{sec:mainproof}

Assume $q$ is large enough (namely larger than $(d(r-1))!$) and let  $p$ be a $d$-dimensional sequence of length $N$ in strong general position
Let $(1,p)$ be the $(d+1)$ dimensional sequence that we get by adding one more sequence, the constant $1$ sequence, as the first coordinate.
Using the results in Section~\ref{sec:orderdominant} (and choosing $N$ suitably large) we find a super-dominant (and then ordered and $q$-increasing)
subsequence $(1,p'_{i_1}),(1,p'_{i_2}),\ldots,(1,p'_{i_n})$ of length $n=(r-1)(d+1)+1$ of $(1,p)$. The all $1$ sequence may not be the first coordinate anymore but it is still the all $1$ sequence.
We need a simple fact.

\begin{claim} Assume $\mathcal A=\{A_1,\ldots,A_r\}$ is a proper partition of $[n]$. Then the partition induced by $\mathcal A$ on the set $\{p_{i_1},\ldots,p_{i_n}\}$ is a Tverberg partition if and only if the one induced on $\{p'_{i_1},\ldots,p'_{i_n}\}$ is a Tverberg partition.
\end{claim}

\begin{proof} Assume $z$ is the Tverberg point of the partition induced by  $\mathcal A$ on the $p_{i_1},\ldots,p_{i_n}$. So with suitable coefficients $\al_j \ge 0$
\[
(1,z)^T=\sum_{j \in A_m}\al_j(1,p_{i_j})^T \mbox{ for every }m \in [r].
\]
Let $T$ be the linear transformation carrying the sequence $(1,p_{i_j})$ to the super-dominant sequence $(1,p'_{i_j})$. Set $(1,z')^T=T^{-1}(1,z)^T$. Applying $T^{-1}$ to the above equation we get
\[
(1,z')^T=\sum_{j \in A_m}\al_j(1,p'_{i_j})^T \mbox{ for every }m \in [r]
\]
showing that the induced partition on $\{p'_{i_1},\ldots,p'_{i_n}\}$ is also a Tverberg partition. The proof in the opposite direction is analogous.
\end{proof}

This means that for the proof of Theorem~\ref{th:univ-tverberg} it suffices to work with the super-dominant sequence $(1,p')$. For simpler notation we denote $(1,p')$ by $(1,p)$ from now on. Recall that the all $1$ row may not be the first row.

In this section we prove half of Theorem~\ref{th:univ-tverberg}, namely the following result.

\begin{theorem}
Let $q>(d(r-1))!$. If $(1,p)$ is a super-dominant $q$-increasing sequence on $n=(r-1)(d+1)+1$ elements,
then every rainbow partition of the set $\{ p_1, \dots, p_n \}$ is a Tverberg partitions.
\end{theorem}

\begin{proof}
Let $A_1,\dots, A_r$ be a rainbow partition. Recall that a rainbow partition satisfies $|A_m\cap B_s|=1$ for every $m\in [r]$ and $s\in [d+1]$,
 where block $B_s$ is just the set $\{(s-1)(r-1)+1,\ldots,s(r-1)+1\}$.

Because the partition is rainbow, it is easy to find the dominating monomial of $\det (M_{\ell})$ for
any fixed $\ell \in [n]$. This is what is explained next.
We define $R_1, \dots, R_{d+1}$ to be the partition of $[n]-\{\ell\}$ into $d+1$ parts each of size $r-1$ such that $R_1$ is the
first $r-1$ elements, $R_2$ is the second $r-1$ elements and so on till $R_{d+1}$ are the largest $r-1$ elements.
Each $R_s$ is a subset of the block $B_s$, moreover they are either the leftmost $r-1$ elements or the rightmost $r-1$ elements
of $B_s$ except when $\ell\in B_s$ in which case $R_s=B_s\setminus\{\ell\}$.
As $|R_s|=r-1$ and $|\mathcal A|=r$ each $R_s$ contains exactly one element from every $A_m$ except one, namely from the unique color class missing from $R_s$.
To define $G_{w_{\ell}}$ we let the elements in row $s$ be exactly $R_s$ with each $i\in R_s$ in the column of its color class.
Let $m$ be the missing color class, that is, the color class of the single element in $B_s \setminus R_s$. Then $G_w(s,m)=\dd{z}{s}{}$. This defines the monomial $w_{\ell}$.
If row $s$ is not the all $1$ row, then the factor $-1$ in $w_{\ell}$ is the entry of $M_{\ell}$ sitting in the column $\dd{z}{s}{}$ in the group of rows of color $m$.
If row $s$ is the all $1$ row, then the factor $1$ in $w_{\ell}$ is the unique $1$ in column $\ell$ of $M_{\ell}$ in the group of rows of $m$.

\begin{claim} In $\det (M_{\ell})$ the dominant monomial is $w_{\ell}$.
\end{claim}

\begin{proof} Indeed, for any fixed $t \in [d]$ every element $i$ in row $t$ or above in $G_{w_{\ell}}$ is smaller than any element $j$ below row $t$.
Thus no $z$-switch can create a larger $w'$. Then by Lemma~\ref{l:switch} $w_{\ell}$ is the dominant monomial.
\end{proof}

Let $1\leq \ell_1< \ell_2 \leq n$ be two consecutive elements of the same color class, $A_1$, say.
Let $x$ be the unique element that is the intersection of blocks $B_s$ and $B_{s+1}$,
where $s \in [d]$ is given by $\ell_1\in B_s$ and $\ell_2 \in B_{s+1}$ with $\ell_1 < x< \ell_2$.
For simpler writing set $w_i=w_{\ell_i}$ and $G_i=G_{w_i}$ and $M_i=M_{w_i}$ for $i=1,2$.
$G_i$ represents a transversal $M_i$, that is, one (non-zero) entry from every row and every column of $M_i$.
The elements of these two transversals are the same for $M_1$ and $M_2$  everywhere with the possible exceptions in columns $\ell_1,x,\ell_2,\dd{z}{s}{}$, and $\dd{z}{s+1}{}$.
We distinguish three cases depending on which row is the all $1$ sequence. It could be $s$, $s+1$ or some other row.

\begin{table}[t]\label{fig:4}
	\begin{center}
		\begin{tabular}{c|c|c|c|c|}
			& $\ell_1$ & $x$&$\ell_2$ & $\dd{z}{s+1}{}$ \\
			\hline
		    s&1,2& & &  \\
			\hline
            s+1& & & 1&2   \\
			\hline
			s& &1 & 2&  \\
			\hline
			s+1& &2 & &1  \\
			\hline
		\end{tabular} \hspace{.5in}
		\begin{tabular}{c|c|c|c|c|}
	& $\ell_1$ & $x$&$\ell_2$ & $\dd{z}{s}{}$ \\
	\hline
	s&2& & & 1\\
	\hline
	s+1& & &1,2 &  \\
	\hline
	s& &1 & &2  \\
	\hline
	s+1& 1&2 & &  \\
	\hline
\end{tabular} \\ \vspace{.2in}
		\begin{tabular}{c|c|c|c|c|c|}
	& $\ell_1$ & $x$&$\ell_2$ & $\dd{z}{s}{}$ & $\dd{z}{s+1}{}$ \\
	\hline
	s&2& & & 1 & \\
	\hline
	s+1& & &1 & & 2 \\
	\hline
	s& &1 & & 2&  \\
	\hline
	s+1& & 2 & & & 1\\
	\hline
	t& 1 &  & 2& & \\
	\hline
\end{tabular}
	\end{center}
	\caption{the position of the elements of the dominant monomial in $M_{\ell_1}$ and $M_{\ell_2}$ (right)}
\end{table}

{\bf Case (i)} when row $s$ is the all $1$ sequence. In this case there is no column corresponding to $\dd{z}{s}{}$.
Thus the two transversals only differ in columns $\ell_1,x,\ell_2,\dd{z}{s+1}{}$. The first matrix in Table 3 depicts the positions of
the four columns and rows of $M_i$ where changes occur. The first two rows are rows $s,s+1$ in the group of $(d+1)$ rows of color $1$, the color class of $\ell_1$ and $\ell_2$.
The next two rows are again rows $s,s+1$ in the group rows of color $m$ where $m$ is the color class of $x$. Direct checking shows that the $1$s resp. $2$s
in the matrix give the positions of the transversals corresponding to $w_1$ and $w_2$.
One can read from the two $4\times4$ matrices that their determinants have the same sign.

As an example we explain how to check entry $1$ in row $1$ and column $1$ of the first matrix in Table 3.
As row $s$ is the all $1$ row and color $1$ (the color of $\ell_1$) is the missing color class from $R_s$, $G_1(s,1)=\ell_1$. The
corresponding entry in $M_1$ sits in column $\ell_1$ and in row $s$ which is the $s$th row in the group of rows of color $1$.
Another example is entry $2$ in the fourth row and second column. Let $m$ be the color class of $x$. As $G_2(s+1,m)=x$, the corresponding entry in $M_2$ lies
in row $s+1$ of the group of rows of color $m$ and in column $x$.

{\bf Case (ii)} when row $(s+1)$ is the all $1$ sequence. There is no column corresponding to $\dd{z}{s+1}{}$.
The second matrix in Table 3 shows the positions of the four columns and rows where there are changes.
One can read from the two matrices that their determinants have the same sign.

{\bf Case (iii)} when row $t$ is the all $1$ row and $t\notin \{s, s+1\}$.
Also the missing color in $R_t$ could be $1$ or the color of $x$ or any other color.
The row related to that color and coordinate is the same though for both $\ell_1$ and $\ell_2$.
The last matrix in Table 3 depicts the positions of the five rows and columns where changes occur.
As in Case (i) and (ii) the first two rows are rows  $s,s+1$ in the group of rows of color $1$ and the next two rows are again rows $s,s+1$ among the rows the color class of $x$.
The fifth row corresponds to the all $1$ row among the rows of the appropriate color class. Observe that the position of the
fifth row here could be either the first row, the third row or the fifth row. But either way that does not change the sign of
the determinant. Again we can see that the determinants of the two matrices have the same sign.
\end{proof}

\section{Proof of Theorem~\ref{th:univ-tverberg}, last part}\label{sec:mainproofcont}

Finally we prove the second half of Theorem~\ref{th:univ-tverberg}. We need a simple lemma.

\begin{lemma} Assume that there are consecutive elements $\ell_1$ and $\ell_2$ of a color class such that
$G_{w_1}$ and $G_{w_2}$ have their $z$ entries in the exact same position. Then the corresponding
transversals in $M_1$ and $M_2$ have different signs.
\end{lemma}

\begin{proof} Assume both $\ell_1,\ell_2 \in A_1$ and row $s$ is the all $1$ row in $(1,p)$.  Then $G_1(s,1)=\dd{z}{s}{}$
and $G_1(m,1)=\ell_2$ for some $m$. As the $z$ coordinates are at the same positions $G_2(s,1)=\dd{z}{s}{}$. The increasing order
in columns rule implies that $G_2(m,1)=\ell_1$. So the only difference between the corresponding transversals occurs
in the $2\times2$ submatrix in columns $\ell_1$ and $\ell_2$ and in row $s$ and the one corresponding to $m$.
A simple checking shows that the corresponding transversals have different signs.
\end{proof}

\begin{theorem} Let $A_1, \dots, A_r$ be a partition that is not rainbow. Then there are
consecutive elements $\ell_1$ and $\ell_2$ of a color class such that $G_1$ and $G_2$
have their $z$ entries in the exact same position.
\end{theorem}

\begin{proof}
We use recursion on $d$ as in Section~\ref{sec:finddom} and rely on the algorithm of Lemma~\ref{l:split}.

Let $\tau\in [d]$ be maximal with respect to $\vdash$. Let $U$ be the first $\tau(r-1)+1$ elements of $[n]$ and let
$V$ be the last $(d+1-\tau)(r-1)+1$ elements of $[n]$. Observe that $U$ and $V$ have exactly one element in common,
namely $x =\tau(r-1)+1$. Recall the notation $A_m^Z=A_m\cap Z$ where $Z \subset [n]$.

We begin with the case $d=1$. Then all color classes have 2 elements except one, the first color class say, that has exactly one element.
If $A_1=\{x\}$, then there must be a class $A_m=\{\ell_1,\ell_2\}$ with $\ell_1,\ell_2<x$ because the partition is not rainbow. One can check easily that the $z$ entries are
at the same position in $G_1$ and  $G_2$. If $A_1=\{y\}$ and $y>x$ (say), then there is a class $A_m=\{\ell_1,\ell_2\}$ with $\ell_1,\ell_2<y$.
The same checking shows that the $z$ entries are at the same positions in both cases.

We suppose now that $d>1$ and that the statement holds in all dimensions less than $d$. We distinguish several cases.

{\bf Case (i)} when $|A_m^U|\le \tau$ and  $|A_m^V|\le d+1-\tau$ for every $m \in [r]$. Then one of the two
partitions induced on $U$ and $V$ must be not rainbow as otherwise the original partition is rainbow. Assume the partition induced by $U$ is not
rainbow. Set $Y=V\setminus \{x\}$ and find the dominant filling of the matrix $G^{Y}$ consisting of the last $d+1-\tau$ rows of $G$ for the partition
$A_m^{Y}$, $m \in [r]$. The recursive algorithm of Lemma~\ref{l:split} gives the position of the $z$-coordinates in $G^Y$
independently of $U$. Then recurse on $U$, that is, on the partition $A_m^U$, $m\in [r]$ which is in dimension $\tau$ now.
We find the two elements $\ell_1, \ell_2$ in $U$ with desired property.

Assume now that we are not in Case (i). Define the excess of $U$ resp. $V$ in $A_m$ as $e(U,m)=||A_m\cap U|-\tau|_+$ and $e(V,m)=||A_m\cap V|-(d+1)-\tau)|_+$.
Set $e(U)=\sum_1^r e(U,m)$ and $e(V)=\sum_1^r e(V,m)$.

{\bf Case (ii)} when $e(U)\ne e(V)$. Assume without loss of generality that $e(U)>e(V)\ge 0$. Then $e(U,m)>0$ for  some $m \in [r]$.
We claim that the first two elements $\ell_1$ and $\ell_2$ of $A_m$ have the required property. We prove this using the algorithm in the proof of Lemma~\ref{l:split} on the sets $[n]\setminus {\ell_i}$ with $\tau$ as the  $\vdash$-maximal element of $[d]$. We define $X=X_i=U\setminus \{\ell_i\}$ for $i=1,2$ and $Y=V\setminus\{x\}$. (We will just use $X$ for $X_1$ and $X_2$ with no confusion emerging.). They have the right sizes for the algorithm, and $e(X)=e(U)-1$ because $\ell_i\in X$ and $e(Y) \leq e(V)$. Thus $e(X)\ge e(Y)$. Since $e(U,m)>\tau$, $|A_m^X|\ge \tau$ with either $\ell_1$ or $\ell_2$ missing. After the exchange we get $X',Y'$ with $|X'|=|X|-e(X)+e(Y)$ and $|Y'|=|Y|-e(Y)+e(X)$, so $|X'|\le |Y'|$. If they are equal, then the algorithm is over and $A_m^X=A_m^{X'}$. Thus $|A_m^{X'}|=\tau$ meaning that no $z$ entry will appear in the first $\tau$ rows of column $m$. During the recursion, the algorithm does not see the difference between whether $\ell_1$ or $\ell_2$ is in the first position of column $m$. So the $z$ entries go to the same position in both cases.

The case $|X'|<|Y'|$ is similar. Then we have to push up some pushable elements from $Y'$ to $X'$ but $A_m^{X'}$ will not change as there is no room to push anybody there. Thus $A_m^X=A_m^{X''}$ and the previous argument works. So all the $z$ entries have to be in the same positions in $G_1$ and $G_2$.

{\bf Case (iii)} when $e(U)=e(V)>0$. We can assume without loss of generality that  $x=\tau(r-1)+1$ is in the first color class of, that is, $x \in A_1$. We say that $x$ is in excess if
either $|(U\setminus \{x\})\cap A_1|\ge \tau$ or $|(V\setminus \{x\})\cap A_1|\ge d+1-\tau$. We distinguish two sub-cases depending on whether $x$ is in excess or not.

{\bf Case (iii-a)} when $x$ is in excess, $|(V\setminus \{x\})\cap A_1|\ge d+1-\tau$, say. Then one can't have  $|(U\setminus \{x\})\cap A_1|\ge \tau$ as well
since that would imply $|A_1|>d+1$. Thus there is a column $m\ne 1$ such that $e(U,m)>0$. We claim again that the two smallest elements, $\ell_1$ and $\ell_2$, of $A_m$
have the required property. The proof is almost identical to that of Case (ii). Define $X=U\setminus \{\ell_i\}$ and $Y=V\setminus \{x\}$.
Here $e(X)=e(U)-1$ and $e(Y)=e(V)-1$ so $e(X)=e(Y)\ge 0$. If $e(X)=e(Y)=0$ then the algorithm of Lemma~\ref{l:split} stops with the pair $X,Y$, and $|A_m^X|=\tau$.
The argument used in Case (ii) works. If $e(X)=e(Y)>0$, then after the exchange the algorithm stops with the pair $X',Y'$. Again $A_m^X=A_m^{X'}$ and we are finished the same way as in Case (ii).

{\bf Case (iii-b)} when $x$ is not in excess. This implies that $e(U,1)=e(V,1) = 0$ and so $|A_1|\le d$. As $e(U)>0$ there is some $m\ne 1$
with $e(U,m)>0$. Let $\ell_1$ and $\ell_2$ be the two smallest elements of $A_m$. It follows that both $\ell_i \in U$.
We claim again that $\ell_1$ and $\ell_2$ have the required property. Set again $X=U\setminus \{\ell_i\}$ and $Y=V\setminus \{x\}$.
Then $e(X)=e(U)-1$ and $e(Y)=e(V)$ so $e(Y)-e(X)=1$. 
Since $x$ is pushable and the smallest element of $Y$, therefore it is the only element that is pushed.
The argument is finished the same way as in Case (iii-a).
\end{proof}

We remark finally that the proof of Theorem~\ref{th:univtverberg} follows from that of Theorem~\ref{th:univ-tverberg}. Just the super-dominant $q$-increasing subsequence of $a:[N]\to \rr^d$, whose existence is guaranteed by the results in Section~\ref{sec:orderdominant}, has to be of size $m$ instead of $n$. That can be achieved using Ramsey theory again with a suitably larger $N$.

\section{Acknowledgment} This research was supported by ERC Advanced Research grant 267165 (DISCONV).

{\sc Attila P\'or}\\[-1.5mm]
{\footnotesize Department of Mathematics}\\[-1.5mm]
{\footnotesize Western Kentucky University}\\[-1.5mm]
{\footnotesize  1906 College Heights Blvd. \#11078}\\[-1.5mm]
{\footnotesize   Bowling Green, KY 42101, USA}
\\[-1.5mm]   {\footnotesize e-mail: {\tt  attila.por@wku.edu}}

\end{document}